\documentclass[12pt,reqno]{amsart}

\usepackage{amssymb}

\newcommand{\cA}{{\mathcal A}}
\newcommand{\cN}{{\mathcal N}}

\newcommand{\oA}{{\overline A}}
\newcommand{\oB}{{\overline B}}
\newcommand{\oC}{{\overline C}}
\newcommand{\oS}{{\overline S}}

\newcommand{\del}{\delta}
\newcommand{\sig}{\sigma}
\newcommand{\zet}{\zeta}
\renewcommand{\phi}{\varphi}

\newcommand{\longc}{,\dotsc,}
\newcommand{\longp}{+\dotsb+}

\newcommand{\longu}{\cup\dotsb\cup}

\renewcommand{\o}[1]{{\overline{#1}}}

\newcommand{\est}{\varnothing}
\newcommand{\seq}{\subseteq}
\newcommand{\stm}{\setminus}

\renewcommand{\>}{\rangle}

\theoremstyle{plain}
\newtheorem{claim}{Claim}
\newtheorem{lemma}{Lemma}
\newtheorem{theorem}{Theorem}
\newtheorem{corollary}{Corollary}
\newtheorem{proposition}{Proposition}

\newtheorem{alphaclaim}{Claim}

\newtheorem{primelemma}{Lemma}

\newtheorem{primetheorem}{Theorem}

\newtheorem*{mainlemma}{Main Lemma}

\newcommand{\refc}[1]{\ref{c:#1}}
\newcommand{\refl}[1]{\ref{l:#1}}
\newcommand{\refm}[1]{\ref{m:#1}}
\newcommand{\reft}[1]{\ref{t:#1}}
\newcommand{\refp}[1]{\ref{p:#1}}
\newcommand{\refs}[1]{\ref{s:#1}}
\newcommand{\refb}[1]{\cite{b:#1}}
\newcommand{\refe}[1]{\eqref{e:#1}}

\title[Symmetric Kneser's Theorem]%
  {Symmetric Kneser's Theorem \\ with trios and $3$-Transform}

\author[]{David J. Grynkiewicz}
\email{diambri@hotmail.com}
\address{Department of Mathematical Sciences, University of Memphis,
  Memphis TN 38152, USA}

\author[]{Vsevolod F. Lev}
\email{seva@math.haifa.ac.il}
\address{Department of Mathematics, The University of Haifa at Oranim,
  Tivon 36006, Israel}

\begin{document}
\baselineskip=16pt

\begin{abstract}
We give a new equivalent restatement and a new proof in terms of trios to the
classical Kneser's theorem. In the finite case, our restatement takes the
following, particularly symmetric shape: if $A$, $B$, and $C$ are subsets of
a finite abelian group $G$ such that $A+B+C\ne G$, then, denoting by $H$ the
period of the sumset $A+B+C$, we have
  $$ |A|+|B|+|C| \le |G|+|H|. $$
The proof is based on an extension of the familiar Dyson transform onto set
systems containing three (or more) sets.
\end{abstract}

\maketitle

\section{Introduction: Kneser's Theorem and Trios}

For a subset $S$ of an abelian group, let $\pi(S)$ denote the \emph{period
(stabilizer)} of $S$; that is, $\pi(S)$ is the subgroup consisting of all
those group elements $g$ with $S+g=S$.

One of the most basic yet robust results in additive combinatorics, Kneser's
theorem, is standardly formulated as follows.
\begin{theorem}[Kneser \cite{b:kn1,b:kn2}]\label{t:K}
If $A$ and $B$ are finite subsets of an abelian group, then
  $$ |A+B| \ge |A|+|B|-|\pi(A+B)|. $$
\end{theorem}

The goal of the present paper is to restate Kneser's theorem in a
``symmetric'' form, and give our restatement an independent proof in terms of
trios.

Following \refb{bdm}, by a \emph{trio} in an abelian group $G$, we mean a
triple $(A,B,C)$ of non-empty subsets of $G$ such that $A+B+C\ne G$ and each
of $A,B$, and $C$ is either finite or co-finite in $G$. Since the sum of two
co-finite subsets of an infinite group is the whole group, every trio can
have at most one infinite component.

The \emph{deficiency} of a trio $(A,B,C)$, denoted $\del(A,B,C)$, is defined
as follows. If $A,B$, and $C$ are all finite, while the underlying group $G$
is infinite, then we let $\del(A,B,C)=-\infty$; this case is in fact of no
importance as we are interested in ``large'' trios. Otherwise, if $G$ is
infinite, then exactly one of the sets $A$, $B$, and $C$ is co-finite, while
the other two are finite, and if $G$ is finite, then all these three sets are
both finite and co-finite. In either case, we can rename the sets involved so
that $A$ is co-finite while $B$ and $C$ are finite, and with this assumption,
we let $\del(A,B,C)=-|\oA|+|B|+|C|$, where for a set $S\seq G$, we write
$\oS:=G\stm S$. Notice that, if $G$ is finite, then
\begin{multline*}
  -|\oA|+|B|+|C| = |A|-|\oB|+|C| = |A|+|B|-|\oC|  = |A|+|B|+|C|-|G|,
\end{multline*}
showing that deficiency is well-defined in this case.%
\footnote{The last expression suggests that \emph{redundancy} might be a more
  intuitive term than \emph{deficiency}; however, we stick with the
  terminology of \refb{bdm}.}

We can now present our restatement of Kneser's theorem
(cf.~\cite[Theorem~3.6]{b:bdm}).
\begin{theorem}\label{t:Krest}
For any trio $(A,B,C)$, we have $\del(A,B,C) \le |\pi(A+B+C)|$.
\end{theorem}

Observe that, in the finite case, Theorem~\reft{Krest} can be given a
particularly simple shape.
\begin{primetheorem}
If $A$, $B$, and $C$ are subsets of a finite abelian group $G$ such that
$A+B+C\ne G$, then
  $$ |A|+|B|+|C| \le |G| + |\pi(A+B+C)|. $$
\end{primetheorem}

We keep using the convention that, for a set $S$, the complement of $S$ in
the underlying group is denoted by $\oS$.

The equivalence of
Theorems~\reft{K} and \reft{Krest} is easy to establish using the following
simple fact.
\begin{claim}\label{m:complement}
For any subset $S$ of an abelian group, we have $\pi(S-\oS)=\pi(S)$.
Moreover, if $S$ is either finite or co-finite, then indeed
  $$ S-\oS = \oS-S = \o{\pi(S)}. $$
\end{claim}

\begin{proof}
For a group element $g$, we have $g\notin\pi(S)$ if and only if either
$S+g\nsubseteq S$, or $S-g\nsubseteq S$. The former relation is equivalent to
$g\in\oS-S$, and the latter to $g\in S-\oS$. We thus conclude that
$\o{\pi(S)}=(S-\oS)\cup(\oS-S)$, whence, in view of $S-\oS=-(\oS-S)$,
  $$ \pi(S)=\pi(\overline{\pi(S)})=\pi\big((S-\oS)\cup(\oS-S)\big)
                                            \ge\pi(S-\oS) \ge \pi(S), $$
implying the first assertion.

Furthermore, if $S$ is either finite or co-finite, then $g\notin\pi(S)$ if
and only if $S+g\not\seq S$, which yields $\o{\pi(S)}=\oS-S$. Switching the
roles of $S$ and $\oS$ and observing that $\pi(S)=\pi(\oS)$, we get
$\o{\pi(S)}=S-\oS$, and the second assertion follows.
\end{proof}

To derive Theorem~\reft{Krest} from Theorem~\reft{K}, assume that $(A,B,C)$
is a trio with $B$ and $C$ finite and $A$ co-finite, and fix a group element
$g\notin A+B+C$; we then have $B+C\seq g-\oA$, whence, by Theorem~\reft{K},
  $$ \del(A,B,C) \le |B|+|C|-|B+C| \le |\pi(B+C)| \le |\pi(A+B+C)|. $$

Conversely, assuming Theorem~\reft{Krest}, and given finite, non-empty
subsets $A$ and $B$ of an abelian group $G$, let $C:=-\o{A+B}$. If $C=\est$,
then $A+B=G$ and $|A+B|\ge|A|+|B|-|\pi(A+B)|$ is immediate. If $C\ne\est$,
then $0\notin A+B-\o{A+B}=A+B+C$, showing that $(A,B,C)$ is a trio, and from
Theorem~\reft{Krest} and Claim~\refm{complement} (applied with $S=A+B$), it
follows that
  $$ |A|+|B|-|A+B| = \del(A,B,C) \le |\pi(A+B+C)| = |\pi(A+B)|. $$

We have shown that Theorems~\reft{K} and~\reft{Krest} are equivalent in the
sense that each of them follows easily from the other one. Our goal now is to
give Theorem~\reft{Krest} an independent, ``symmetric'' proof. As preparation
steps, in the next section we collect some background facts about trios, and
in Section~\refs{trans} we develop a multiple-set transform, the basic tool
employed in our proof. The proof itself is then presented in
Sections~\refs{mainlemma} and~\refs{proof}. Concluding remarks are gathered
in Section~\refs{concrem}.

\section{Trios}\label{s:trios}

In this section we provide the background about trios needed for the proof of
Theorem~\reft{Krest}. Most of the material here is contained, in this or
another form, in \refb{bdm}.

Refining the definition from the previous section, for an abelian group $G$
and an element $g\in G$, we say that a triple $(A,B,C)$ of non-empty subsets
of $G$ is a \emph{$g$-trio} if $g\notin A+B+C$ and each of $A,B$, and $C$ is
either finite or co-finite in $G$.

The \emph{period} of the trio $(A,B,C)$ is the period of the sumset $A+B+C$.
Since this sumset is either a finite or a co-finite subset of the underlying
group, the period of a trio is always finite. The trio $(A,B,C)$ is
\emph{aperiodic} if its period is the trivial subgroup, and \emph{periodic}
otherwise. It is easily verified that if $(A,B,C)$ is a trio in an abelian
group $G$, and $H=\pi(A+B+C)$, then the images of $A$, $B$, and $C$ under the
canonical homomorphism $G\to G/H$ form an aperiodic trio in the quotient
group $G/H$.

Clearly, if $(A,B,C)$ is a $g$-trio in an abelian group $G$, then
$(A-a,B-b,C-c)$ is a $(g-a-b-c)$-trio in $G$ for any $a,b,c\in G$, and both
trios share the same period.

As a direct consequence of Claim~\refm{complement}, we have the following
corollary.

\begin{corollary}\label{c:C-oC}
Suppose that $(A,B,C)$ is a $g$-trio in an abelian group $G$. If
$C=g-\o{A+B}$, then, letting $H:=\pi(A+B+C)$, we have $A+B+C=G\stm(g+H)$ and
$\pi(C)=H$.
\end{corollary}

Furthermore, we have the following lemma.
\begin{lemma}\label{l:khinchin}
Suppose that $(A,B,C)$ is a $g$-trio and let $C':=g-\o{A+B}$. Then $(A,B,C')$
is also a $g$-trio, $C\seq C'$, and $\pi(A+B+C')\le\pi(A+B+C)$.
\end{lemma}

\begin{proof}
The definition of $C'$ readily implies that $g\notin A+B+C'$, and that $C'$
is either finite or co-finite (the latter follows from finiteness or
co-finiteness of $A$ and $B$); consequently, $(A,B,C')$ is a $g$-trio. Since
$(A,B,C)$ is a $g$-trio, we have $g\notin A+B+C$, whence $C\seq
g-\o{A+B}=C'$. Finally, by Corollary~\refc{C-oC},
  $$ \pi(A+B+C') = \pi(C') = \pi(A+B) \le \pi(A+B+C). $$
\end{proof}

The trio $(A,B,C)$ is \emph{contained} in the trio $(A',B',C')$ if $A\seq
A'$, $B\seq B'$, and $C\seq C'$; in this case, the former trio is also said
to be a \emph{subtrio} of the latter, and the latter a \emph{supertrio} of
the former.

We say that $(A,B,C)$ is a \emph{maximal} $g$-trio if (in addition to being a
$g$-trio) it is not properly contained in any other $g$-trio; that is, for
any $g$-trio $(A',B',C')$ with $A\seq A'$, $B\seq B'$, and $C\seq C'$, we
actually have $A=A'$, $B=B'$, and $C=C'$. By Lemma~\refl{khinchin}, for
$(A,B,C)$ to be a maximal $g$-trio, it is necessary and sufficient that
$A=g-\o{B+C}$, $B=g-\o{C+A}$, and $C=g-\o{A+B}$. Hence, by
Corollary~\refc{C-oC}, if $(A,B,C)$ is a maximal $g$-trio, then, letting
$H:=\pi(A+B+C)$, we have $\pi(A)=\pi(B)=\pi(C)=H$ and $A+B+C=G\stm(g+H)$. In
particular, if $(A,B,C)$ is a maximal \emph{aperiodic} $g$-trio, then
$A+B+C=G\stm\{g\}$.

\begin{lemma}\label{l:gf}
If $(A,B,C)$ is a maximal $g$-trio, then it is in fact a maximal $f$-trio for
each $f\notin A+B+C$.
\end{lemma}

\begin{proof}
Clearly $(A,B,C)$ is an $f$-trio, whence
\begin{equation}\label{e:gf-a}
  A\seq f-\o{B+C},\ B\seq f-\o{C+A},\ \text{and}\ C\seq f-\o{A+B}.
\end{equation}
On the other hand, since $(A,B,C)$ is a maximal $g$-trio,
\begin{equation}\label{e:gf-b}
  A=g-\o{B+C},\ B=g-\o{C+A},\ \text{and}\ C=g-\o{A+B}.
\end{equation}

We now claim that none of the inclusions in \refe{gf-a} is strict; for if we
had, for instance, $C\subsetneq f-\o{A+B}$, then in view of \refe{gf-b} this
would imply $C\subsetneq(f-g)+C$, which is impossible since $C$ is either
finite or co-finite. This shows that $(A,B,C)$ is a maximal $f$-trio.
\end{proof}

With Lemma~\refl{gf} in mind, we can speak about maximal trios without
indicating the specific value of $g$. In addition, Lemma~\refl{gf} shows that
$(A,B,C)$ is a maximal $g$-trio for some group element $g$ if and only if it
is a maximal trio; that is, not properly contained in any other trio.

\begin{lemma}\label{l:astrio}
For a $g$-trio $(A,B,C)$, define  $A':=g-\o{B+C}$, and then subsequently
$B':=g-\o{A'+C}$ and $C':=g-\o{A'+B'}$. Then
\begin{itemize}
\item[i)]   $A\seq A',\ B\seq B'$, and $C\seq C'$;
\item[ii)]  $(A',B',C')$ is a maximal $g$-trio;
\item[iii)] $\pi(A'+B'+C')\le\pi(A+B+C)$.
\end{itemize}
\end{lemma}

\begin{proof}
The fact that $(A',B',C')$ is a $g$-trio containing $(A,B,C)$, and also the
relation $\pi(A'+B'+C')\le\pi(A+B+C)$, follow readily by repeated application
of Lemma~\refl{khinchin}. To see why $(A',B',C')$ is maximal, notice that, if
it is contained in a $g$-trio $(A'',B'',C'')$, then
  $$ A' \seq A'' \seq g-\o{B''+C''}
             \seq g-\o{B'+C'} \seq g-\o{B+C} = A' $$
implying $A''=A'$, and similarly $B''=B'$ and $C''=C'$.
\end{proof}

As it follows from Lemma~\refl{astrio}, every aperiodic trio can be embedded
into a maximal aperiodic trio.

A trio is said to be \emph{deficient} if its deficiency is positive. To
conclude this section, we record the following corollary of
Theorem~\reft{Krest}.
\begin{corollary}\label{c:maxdef}
If $(A,B,C)$ is a maximal, deficient trio, then
  $$ \del(A,B,C)=|\pi(A+B+C)|. $$
\end{corollary}

\begin{proof}
Let $H:=\pi(A+B+C)$ and assume for definiteness that $B$ and $C$ are finite.
By maximality, we have $\pi(A)=\pi(B)=\pi(C)=H$, so that $|B|$, $|C|$ and
$|\oA|$ are all divisible by $|H|$; hence also $\del(A,B,C)$ is divisible by
$|H|$. On the other hand, deficiency gives $\del(A,B,C)>0$, and the
conclusion now follows from Theorem~\reft{Krest}.
\end{proof}

\section{The $n$-transform}\label{s:trans}

In this section we introduce a version of the classical Dyson transform for a
set system potentially involving more than two sets. We call it the
\emph{$n$-transform}, where $n$ can be substituted with the actual number of
sets; thus, the Dyson transform is the $2$-transform, and what we ultimately
need for the proof of Theorem~\reft{Krest} is the $3$-transform.

Let $\cA=(A_\nu)_{\nu\in\cN}$ be a system of subsets of some ground set
$G$. For an integer $i\ge 1$, denote by $\tau_i(\cA)$ the set of all those
elements of $G$ belonging to at least $i$ sets from $\cA$, and let
$\tau(\cA)=(\tau_i(\cA))_{i\ge 1}$. If $\cA$ is finite with, say, $|\cN|=n$,
then the sets $\tau_i(\cA)$ are empty for $i>n$, and we then identify
$\tau(\cA)$ with the finite sequence $(\tau_i(\cA))_{1\le i\le n}$; notice
that, in this case, $\tau_{n}(\cA)=\cap_{\nu\in\cN}A_\nu$, and that we always
have $\tau_1(\cA)=\cup_{\nu\in\cN} A_\nu$.

Although we are interested in the situation where $\cA$ is a finite sequence
of subsets of an \emph{abelian group}, we start with two general
set-theoretic properties of the $n$-transform.

\begin{lemma}\label{l:major}
If $A_1\longc A_n$ are finite sets, then letting
 $(A^*_1\longc A^*_n)=\tau(A_1\longc A_n)$, we have
\begin{equation}\label{e:major}
  |A^*_1|\longp|A^*_k| \ge |A_1|\longp|A_k|\ \; \text{for}\ k\in[1,n]
\end{equation}
with equality for $k=n$. If, indeed, equality holds in \refe{major} \emph{for
each $k\in[1,n]$}, then $A_1\supseteq\dotsb\supseteq A_n$ (whence $A^*_k=A_k$
for all $k\in[1,n]$).
\end{lemma}

\begin{proof}
The equality $|A^*_1|\longp|A^*_n|=|A_1|\longp|A_n|$ follows by observing
that for every element $g$ of the ground set, the number of the sets $A_i$
that contain $g$ is equal to the number of the sets $A^*_i$ containing $g$.
For $1\le k<n$, replacing each of the sets $A_{k+1}\longc A_n$ with the empty
set (which does not affect the sum $|A_1|\longp|A_k|$, and can only make the
sum $|A^*_1|\longp|A^*_k|$ smaller), we reduce the situation to the case
$k=n$ just considered.

To prove the second assertion, we first notice that if equality holds in
\refe{major} for all $k\in[1,n]$, then $|A^*_1|=|A_1|\longc |A^*_n|=|A_n|$,
and then use induction by $n$. The case $n=1$ is immediate, and we assume
therefore that $n\ge 2$. Since $A^*_1=A_1\longu A_n$, from $|A^*_1|=|A_1|$ we
derive that, in fact, $A_1=A^*_1=A_1\longu A_n$, whence $A_k\subseteq A_1$
for all $k\in[2,n]$. Hence, for each $k\in[2,n]$, the set $A^*_k$ consists of
all those elements contained in at least $k-1$ of the sets $A_2\longc A_n$;
that is, $(A^*_2\longc A^*_n)=\tau(A_2\longc A_n)$. By the induction
hypothesis, we have then $A_2\supseteq\dotsb\supseteq A_n$ and the assertion
follows.
\end{proof}

\begin{lemma}\label{l:stab}
For a sequence of sets $(A_1,A_2,\ldots)$ to be stable under the
$n$-transform, it is necessary and sufficient that
 $A_1\supseteq A_2\supseteq\dotsc$.
\end{lemma}

\begin{proof}
Clearly, the condition is sufficient: if $A_1\supseteq A_2\supseteq\dotsc$,
then $\tau_k(A_1,A_2,\ldots)=A_k$ for each $k\ge 1$. It is also necessary
for, in general, if $(B_1,B_2,\ldots)$ is an image of some sequence under the
$n$-transform, then $B_1\supseteq B_2\supseteq\dotsb$.
\end{proof}

We now turn to the properties of the $n$-transform specific to subsets of
abelian groups.

For integers $a_1\longc a_n,b_1\longc b_n$, we write
 $(a_1\longc a_n)\prec(b_1\longc b_n)$ if $(b_1\longc b_n)$ majorizes
$(a_1\longc a_n)$; that is, if $a_1\longp a_k\le b_1\longp b_k$ for each
$k\in[1,n]$, with equality for $k=n$ and strict inequality for at least one
$k\in[1,n-1]$. Notice that, if $(a_1\longc a_n)\prec(b_1\longc b_n)$, then
$(a_1\longc a_n)$ precedes $(b_1\longc b_n)$ also in the lexicographic order
$\prec_\ell$.

\begin{lemma}\label{l:stopcond} 
For any finite subsets $A_1\longc A_n$ of an abelian group, one of the
following holds:
\begin{itemize}
\item[i)]  There exist elements $a_k\in A_k\ (k\in [1,n])$ such that,
    letting $(A^*_1\longc A^*_n):=\tau(A_1-a_1\longc A_n-a_n)$, we have
      $$ (|A_1|\longc|A_n|)\prec(|A^*_1|\longc|A^*_n|). $$
\item[ii)] We have $A_k-A_k\seq\pi(A_{k-1})$ for all $k\in [2,n]$; that
    is, for each $k\in[2,n]$, the set $A_k$ is contained in a coset of
    the period of the set $A_{k-1}$.
\end{itemize}
\end{lemma}

\begin{proof}
By Lemma~\refl{major}, if $(|A_1|\longc|A_n|)\prec(|A^*_1|\longc|A^*_n|)$
does \emph{not} hold for some particular choice of the elements $a_k\in A_k$,
then for each $k\in[2,n]$, we have $A_k-a_k\seq A_{k-1}-a_{k-1}$, and
therefore $A_k-a_k+a_{k-1}\seq A_{k-1}$. If now $a_k\in A_k$ and $a_{k-1}\in
A_{k-1}$ can be chosen arbitrarily, this leads to $A_k-A_k+A_{k-1}\seq
A_{k-1}$, whence $A_k-A_k\seq \pi(A_{k-1})$.
\end{proof}

\begin{lemma}\label{l:basic} 
If $A_1\longc A_n$ are subsets of an abelian group, then letting
$(A^*_1\longc A^*_n)=\tau(A_1\longc A_n)$ we have
  $$ A^*_1\longp A^*_n \seq A_1\longp A_n. $$
\end{lemma}

\begin{proof}
The assertion follows by fixing, for each $g\in A^*_1\longp A^*_n$, a
representation $g=b_1\longp b_n$ with $b_i\in A^*_i$ for each $i\in[1,n]$,
and then recursively choosing indices $i_1,\ldots,i_n\in [1,n]$ so that,
having $i_1\longc i_{k-1}$ found, the next index $i_k$ is chosen to satisfy
$i_k\notin\{i_{1},\ldots,i_{k-1}\}$ and $b_k\in A_{i_k}$. The details are
straightforward.
\end{proof}

For a trio $(A,B,C)$, let $(A^*,B^*,C^*):=\tau(A,B,C)$. As a corollary of
Lemma~\refl{basic}, if $C^*\ne\est$, then $(A^*,B^*,C^*)$ is a trio, too.

We now consider the situation where one of the sets involved can be infinite.
Lemmas~\refl{stab} and~\refl{basic} do not in fact assume finiteness, while
Lemmas~\refl{major} and~\refl{stopcond} extend onto the infinite case as
follows.

\renewcommand{\theprimelemma}{\refl{major}$'$}
\begin{primelemma}\label{l:major-infi}
Suppose that $A_1\longc A_n$ are subsets of some ground set $G$ such that
\mbox{$G\stm A_1$}, $A_2\longc A_n$ are all finite, and let
 $(A^*_1\longc A^*_n):=\tau(A_1\longc A_n)$. Then also $G\stm A^*_1$,
$A^*_2\longc A^*_n$ are all finite, and
\begin{equation}\label{e:major-infi}
  -|G\stm A^*_1|+|A^*_2|\longp|A^*_k|
          \ge -|G\stm A_1|+|A_2|\longp|A_k| \quad\mbox{ for }\; k\in[1,n]
\end{equation}
with equality for $k=n$. If, indeed, equality holds in \refe{major-infi} for
each $k\in[1,n]$, then $A_1\supseteq\dotsb\supseteq A_n$ (whence $A^*_k=A_k$
for all $k\in[1,n]$).
\end{primelemma}

\begin{proof}
The finiteness of $G\stm A^*_1$, $A^*_2\longc A^*_{n}$ is immediate. The
remaining assertions follow from Lemma~\refl{major} by considering the finite
sets $U:=A_2\longu A_n$ and $A_1':=A_1\cap U$, and observing that
  $$ \tau(A_1',A_2\longc A_n)=(U,A^*_2\longc A^*_n), $$
that
  $$ -|G\stm A^*_1|+|G\stm A_1| = |A^*_1\stm A_1|
                                      = |U\stm A_1'| = |U|-|A_1'|, $$
and that $A_1'\supseteq A_2$ implies $A_1\supseteq A_2$.
\end{proof}

\renewcommand{\theprimelemma}{\refl{stopcond}$'$}
\begin{primelemma}\label{l:stopcond-inf} 
For any subsets $A_1\longc A_n$ of an abelian group $G$ such that $G\stm
A_1$, $A_2\longc A_n$ are all finite, one of the following holds:
\begin{itemize}
\item[i)]  There exist elements $a_k\in A_k\ (k\in [1,n])$ such that,
    letting $(A^*_1\longc A^*_n):=\tau(A_1-a_1\longc A_n-a_n)$, we have
      $$ (-|G\stm A_1|,|A_2|\longc|A_n|)\prec(-|G\stm
                                        A^*_1|,|A^*_2|\longc|A^*_n|). $$
\item[ii)] We have $A_k-A_k\seq \pi(A_{k-1})$ for all $k\in[2,n]$; that
    is, for each $k\in[2,n]$, the set $A_k$ is contained in a coset of
    the period of the set $A_{k-1}$.
\end{itemize}
\end{primelemma}

\begin{proof}
The proof is almost identical to that of Lemma~\refl{stopcond}, except that
we now apply Lemma~\refl{major-infi} instead of Lemma~\refl{major}.

By Lemma~\refl{major-infi}, if
 $(-|G\stm A_1|,|A_2|\longc|A_n|)\prec(-|G\stm A^*_1|,|A^*_2|\longc|A^*_n|)$
does \emph{not} hold for some specific choice of $a_k\in A_k$, then for each
$k\in[2,n]$, we have $A_k-a_k+a_{k-1}\seq A_{k-1}$. If now $a_k\in A_k$ and
$a_{k-1}\in A_{k-1}$ can be chosen arbitrarily, this leads to
$A_k-A_k+A_{k-1}\seq A_{k-1}$, meaning that $A_k-A_k\seq\pi(A_{k-1})$ as
$A_k-A_k$ is symmetric.
\end{proof}

Since the order $\prec$ implies the lexicographic order $\prec_\ell$, from
Lemma~\refl{stopcond-inf} we conclude that either there exist elements
$a_k\in A_k$ such that
 $(|G\stm A^*_1|,-|A^*_2|\longc-|A^*_n|)
                                    \prec_\ell(|G\stm A_1|,-|A_2|\longc-|A_n|)$,
where $(A^*_1\longc A^*_n)=\tau(A_1-a_1\longc A_n-a_n)$, or
$A_k-A_k\seq\pi(A_{k-1})$ holds for each $k\in[2,n]$.

\section{The Main Lemma and Overview of the Proof.}\label{s:mainlemma}

The following result is the central ingredient of the proof of
Theorem~\reft{Krest}.
\begin{mainlemma}
Let $(A,B,C)$ be an aperiodic, maximal, deficient trio in an abelian group
$G$ such that $A$ is co-finite. For a triple $(a,b,c)\in G^3$, let
$(A^*,B^*,C^*):=\tau(A-a,B-b,C-c)$ and suppose that there exists $(a,b,c)\in
G^3$ with $C^*\ne\est$ and $(A^*,B^*,C^*)\ne(A-a,B-b,C-c)$. With these
assumptions, choose $(a,b,c)\in G^3$, satisfying the conditions just
mentioned, for which $|G\stm(A^*+B^*+C^*)|$ is smallest possible, and let
$H:=\pi(A^*+B^*+C^*)$. Then
  $$ |(A^*+H)\stm A^*|+|(B^*+H)\stm B^*|+|(C^*+H)\stm C^*| \ge |H|-1. $$
\end{mainlemma}

We actually prove the Main Lemma and Theorem~\reft{Krest} simultaneously,
using induction, as we now proceed to describe.

Clearly, it suffices to prove Theorem~\reft{Krest} only for those trios
$(A,B,C)$ with $A$ co-finite. To every such trio, we associate the quadruple
  $$ \sig(A,B,C):=(|G|,|G\stm A|,-|B|,-|C|), $$
and we denote by $\mathfrak S$ the set of all quadruples that can arise this
way, ordered lexicographically. The proof of the Main Lemma and
Theorem~\reft{Krest} goes by induction on $\sig(A,B,C)$. The induction is
well-founded as $\mathfrak S$ does not contain infinite descending chains
(with respect to the lexicographic order). This follows by observing that,
for $|G\stm A|$ fixed, the set of possible values of $|B|$ and $|C|$ is
finite, because $A+B+C\ne G$ implies $\max\{|B|,|C|\}\le |B+C|\le|G\stm A|$.

As a part of our inductive argument, we now show that, loosely speaking, if
the Main Lemma is true for the trio $(A,B,C)$ with $A$ co-finite, and
Theorem~\reft{Krest} is true for all trios $(A',B',C')$ with $A'$ co-finite
and $\sig(A',B',C')<\sig(A,B,C)$, then Theorem~\reft{Krest} is also true for
the trio $(A,B,C)$.

\begin{proposition}\label{p:ML2K}
Let $(A,B,C)$ be a trio with $A$ co-finite and suppose that
\begin{itemize}
\item[i)] either the assumptions of the Main Lemma fail for $(A,B,C)$ or
    the assertion of the Main Lemma holds for $(A,B,C)$;
\item[ii)] the estimate $\del(A',B',C')\le|\pi(A'+B'+C')|$ holds for all
    trios $(A',B',C')$ with $A'$ co-finite and
    $\sig(A',B',C')<\sig(A,B,C)$.
\end{itemize}
Then $\del(A,B,C)\le|\pi(A+B+C)|$.
\end{proposition}

\begin{proof}
If $(A,B,C)$ is not deficient, then the required estimate
$\del(A,B,C)\le|\pi(A+B+C)|$ is immediate; suppose therefore that $(A,B,C)$
is deficient.

If $(A,B,C)$ is not maximal, then we consider a maximal trio $(A',B',C')$
containing $(A,B,C)$ and satisfying $\pi(A'+B'+C')\le\pi(A+B+C)$,
cf.~Lemma~\refl{astrio}. Since $(A',B',C')$ strictly contains $(A,B,C)$, it
follows that $\sig(A',B',C')<\sig(A,B,C)$ and
$\delta(A,B,C)<\delta(A',B',C')$, and in view of assumption ii), we then
get
  $$ \del(A,B,C) < \del(A',B',C') \le |\pi(A'+B'+C')| \le |\pi(A+B+C)|. $$
Suppose thus that $(A,B,C)$ is maximal.

Let $G$ denote the underlying group, and for a subgroup $K\le G$ denote by
$\phi_K$ the canonical homomorphism from $G$ onto the quotient group $G/K$.

If $(A,B,C)$ is periodic, then, denoting its period by $K$, we get
  $$ \del(A,B,C)\le\del(A+K,B+K,C+K)
                     = |K|\,\del(\phi_K(A),\phi_K(B),\phi_K(C)) \le |K|, $$
with the last inequality following from the assumption ii) in view of
$|G/K|\le|G|$ and
  $$ |\phi_K(G)\stm\phi_K(A)|
                       = |G\stm(A+K)|/|K| \le |G\stm A|/|K| < |G\stm A|. $$
We therefore suppose that $(A,B,C)$ is aperiodic.

If there do not exist $a,b,c\in G$ such that, letting
$(A^*,B^*,C^*):=\tau(A-b,B-b,C-c)$, we have $C^*\ne\est$ and
$(A^*,B^*,C^*)\ne(A-a,B-b,C-c)$, then by Lemma \ref{l:stopcond-inf}, the set
$C$ is contained in a coset of $\pi(B)$, and the set $B$ is contained in a
coset of $\pi(A)$. Since $(A,B,C)$ is aperiodic, this yields $|B|=|C|=1$,
whence
  $$ \del(A,B,C) = -|G\stm A|+|B|+|C| = 2 -|G\stm A|
                                                  \le 1 \le |\pi(A+B+C)|. $$

We thus assume that the trio $(A,B,C)$ satisfies all the assumptions of the
Main Lemma: namely, it is maximal, aperiodic, and deficient, and there exists
$(a,b,c)\in G^3$ such that, letting $(A^*,B^*,C^*):=\tau(A-a,B-b,C-c)$, we
have $C^*\ne\est$ and $(A^*,B^*,C^*)\ne(A-a,B-b,C-c)$. Moreover, we assume
that the triple $(a,b,c)$ is chosen to minimize $|G\stm(A^*+B^*+C^*)|$.
Notice that the condition $(A^*,B^*,C^*)\ne(A-a,B-b,C-c)$ implies
$\sig(A^*,B^*,C^*)<\sig(A-a,B-b,C-c)=\sig(A,B,C)$ by Lemma~\refl{major-infi}.

Let $H:=\pi(A^*+B^*+C^*)$. If $H=\{0\}$, then, by Lemma~\refl{major-infi} and
assumption ii), we have
  $$ \del(A,B,C) = \del(A^*,B^*,C^*) \le |H| = 1 \le |\pi(A+B+C)|. $$
Suppose therefore that $H\ne\{0\}$. In view of $|G/H|\le|G|$ and
\begin{multline}\label{e:indfactor}
  \qquad |\phi_H(G)\stm\phi_H(A^*)| \le |\phi_H(G)\stm\phi_H(A)| \\
                    = |G\stm(A+H)|/|H| \le |G\stm A|/|H| < |G\stm A|, \qquad
\end{multline}
we can apply assumption ii) to the aperiodic trio
$(\phi_H(A^*),\phi_H(B^*),\phi_H(C^*))$ to obtain
\begin{equation}\label{e:factor}
  \del(A^*+H,B^*+H,C^*+H)
                = |H|\, \del(\phi_H(A^*),\phi_H(B^*),\phi_H(C^*)) \le |H|.
\end{equation}
On the other hand, from the Main Lemma,
\begin{equation}\label{e:fromML}
  \del(A^*+H,B^*+H,C^*+H)-\del(A^*,B^*,C^*)\ge |H|-1.
\end{equation}
Comparing \refe{factor} and \refe{fromML} and using Lemma~\refl{major-infi},
we obtain
  $$ \del(A,B,C) = \del(A^*,B^*,C^*) \le 1 \le |\pi(A+B+C)|. $$
\end{proof}

\section{Proof of the Main Lemma and Theorem~\reft{Krest}}\label{s:proof}

\subsection{The set-up and initial observations.}\label{s:setup}
As follows from Proposition~\refp{ML2K}, to establish the Main Lemma and
Theorem~\reft{Krest}, it suffices to prove the former assuming, as an
induction hypothesis, that the latter is true for all ``smaller'' trios.
Having the components of the trio under consideration appropriately
translated, we thus have the following set of assumptions:
\begin{itemize}
\item[i)] $(A,B,C)$ is an aperiodic, maximal, deficient trio in an
    abelian group $G$, with $A$ co-finite.
\item[ii)] the triple $(A^*,B^*,C^*):=\tau(A,B,C)$ satisfies
    $C^*\ne\est$ and $(A^*,B^*,C^*)\ne(A,B,C)$; thus, $(A^*,B^*,C^*)$
    is a trio with $\del(A^*,B^*,C^*)=\del(A,B,C)$ and
    $\sig(A^*,B^*,C^*)<\sig(A,B,C)$ (by Lemmas~\refl{major-infi}
    and~\refl{stopcond-inf}).
\item[iii)] for any $a,b,c\in G$, letting
    $(U^*,V^*,W^*):=\tau(A-a,B-b,C-c)$, we have either $W^*=\est$ or
    $(U^*,V^*,W^*)=(A-a,B-b,C-c)$ or
    $|G\stm(U^*+V^*+W^*)|\ge|G\stm(A^*+B^*+C^*)|$;
\item[iv)] for any trio $(A',B',C')$ with $A'$ co-finite and
    $\sig(A',B',C')<\sig(A,B,C)$, we have
    $\del(A',B',C')\le|\pi(A'+B'+C')|$.
\end{itemize}
We let $H:=\pi(A^*+B^*+C^*)$, and we want to prove that
\begin{equation}\label{e:main2prove}
  |(A^*+H)\stm A^*|+|(B^*+H)\stm B^*|+|(C^*+H)\stm C^*| \ge |H|-1.
\end{equation}
Denote the left-hand side of \refe{main2prove} by $\rho$ and, for a
contradiction, assume that
\begin{equation}\label{e:rho}
  \rho\le |H|-2;
\end{equation}
notice that this implies $|H|\ge 2$.

For an element $x\in G$ and a set $S\seq G$, let $S_x:=S\cap(x+H)$ be the
\emph{$x$-slice of $S$}; thus, if $x\equiv y\pmod H$, then $S_x=S_y$. From
now on, we will write $\o S:=(S+H)\stm S$ for the $H$-complement of $S$.
(Although this is inconsistent with the notation of Section~\refs{trios}, no
confusion should arise as the ``old notation'' will not be used anymore.)
Thus, for instance, $\o{S_x}$ is the complement of $S$ in $x+H$, except that,
if $S$ does not have any elements in this coset, then $\o{S_x}$ is empty:
  $$ \o{S_x} = \begin{cases}
                 (x+H)\stm S &\ \text{if $S_x\ne\est$}, \\
                 \est        &\ \text{if $S_x=\est$};
               \end{cases} $$
as a result, $|\o{S_x}|=t|H|-|S_x|$, where $t=1$ if $S_x\ne\est$, and
$t=0$ otherwise.

Observing that $(A^*_x,B^*_x,C^*_x)=\tau(A_x,B_x,C_x)$, we get
$|A^*_x|+|B^*_x|+|C^*_x|=|A_x|+|B_x|+|C_x|$ by Lemma~\refl{major}.
Consequently, we have
  $$ \rho = \sum_x \rho_x, $$
where $x$ runs over the representatives of all cosets of $H$, and
\begin{align}
  \rho_x &= |(A^*_x+H)\stm A^*|+|(B^*_x+H)\stm B^*|+|(C^*_x+H)\stm C^*|
                                                               \nonumber \\
         &= |\o{A_x^*}| + |\o{B_x^*}| + |\o{C_x^*}| \nonumber \\
         &= t_x|H|-|A_x^*|-|B_x^*|-|C_x^*| \nonumber \\
         &= t_x|H|-|A_x|-|B_x|-|C_x|, \label{e:tx1}
\end{align}
$t_x\in[0,3]$ being the number of non-empty slices among $A_x^*$, $B_x^*$,
and $C_x^*$. In particular, if $A^*_x\ne\est$ (meaning that at least one of
$A_x,B_x$, and $C_x$ is non-empty), then
\begin{equation}\label{e:tx2}
   \rho_x \ge |\o{A^*_x}| \ge |H|-|A_x|-|B_x|-|C_x|,
\end{equation}
and if $C^*_x$ is non-empty (so that also $A^*_x$ and $B^*_x$ are non-empty),
then
\begin{equation}\label{e:tx3}
  \rho_x \ge 3|H|-|A_x|-|B_x|-|C_x| = |\o{A_x}| + |\o{B_x}| + |\o{C_x}|.
\end{equation}

We say that a subset $S$ of an $H$-coset is \emph{partial} if $0<|S|<|H|$,
and is \emph{full} if $|S|=|H|$.

Since $(A,B,C)$ is maximal and aperiodic, there is a unique element of $G$
not lying in $A+B+C$ (see a remark before the statement of Lemma~\refl{gf});
we denote this element by $g_0$, so that $A+B+C=G\stm\{g_0\}$. Notice that,
for every $x\in G$ with $A_x$ partial, there exist $y,z\in G$ with both $B_y$
and $C_z$ partial and $x+y+z\equiv g_0\pmod H$. Indeed, otherwise, for any $y$
and $z$ with $x+y+z\equiv g_0\pmod H$, we would have either $B_y=\est$ or
$C_z=\est$; this would lead to $g_0\notin (x+H)+B+C$ and consequently
$(A\cup(x+H),B,C)$ would also be a trio, contradicting the maximality of
$(A,B,C)$. Similar remarks apply to the situation where $B_y$ or $C_z$ is
partial for some $y,z\in G$. This observation will be used repeatedly in the
proof.

Recall that, for a subgroup $K\le G$, by $\phi_K$ we denote the canonical
homomorphism from $G$ onto $G/K$.

\begin{lemma}\label{l:emptystar}
\hskip 0in 
\begin{itemize}
\item[i)] If $A^*_x=\est$ for some $x\in G$, then there exist slices
    $B^*_y,\,C^*_z\ne\est$ with  $x+y+z\equiv g_0\pmod H$.
\item[ii)]  If $B^*_y=\est$ for some $y\in G$, then there exist slices
    $A^*_x,\,C^*_z\ne\est$ with  $x+y+z\equiv g_0\pmod H$.
\item[iii)]  If $C^*_z=\est$ for some $z\in G$, then there exist slices
    $A^*_x,\,B^*_y\ne\est$ with  $x+y+z\equiv g_0\pmod H$.
\end{itemize}
\end{lemma}

\begin{proof}
We prove the first assertion only; the other two follow in an identical way.

The key observation is that the aperiodic trio
$(\phi_H(A^*),\phi_H(B^*),\phi_H(C^*))$ is maximal: otherwise by
Lemma~\refl{astrio} it would be properly contained in an aperiodic maximal
trio $(U,V,W)$ to which the induction hypothesis applies in view of
$|\phi_H(G)\stm U|\le|\phi_H(G)\stm\phi_H(A^*)|$ and~\refe{indfactor}. This
would lead to
  $$ \del(\phi_H(A^*),\phi_H(B^*),\phi_H(C^*))
                                      < \del(U,V,W) \le |\pi(U+V+W)| = 1, $$
and thus to
  $$ \del(A,B,C) = \del(A^*,B^*,C^*)
                      = H\del(\phi_H(A^*),\phi_H(B^*),\phi_H(C^*)) \le 0, $$
contrary to the deficiency assumption. Now, the maximality of
$(\phi_H(A^*),\phi_H(B^*),\phi_H(C^*))$ shows that $(A^*\cup(x+H),B^*,C^*)$
is \emph{not} a trio, which readily implies the assertion.
\end{proof}

\begin{lemma}\label{l:box}
Let $(A_x,B_y,C_z)$ be a triple of slices with $x+y+z\equiv g_0\pmod H$.
\begin{itemize}
\item[i)]   If $A_x\ne\est$, then $|B_y|+|C_z|\le|H|$;
\item[ii)]  if $B_y\ne\est$, then $|C_z|+|A_x|\le|H|$;
\item[iii)] if $C_z\ne\est$, then $|A_x|+|B_y|\le|H|$.
\end{itemize}
\end{lemma}

\begin{proof}
By the pigeonhole   principle, from $|B_y|+|C_z|>|H|$ we would get
$B_y+C_z=y+z+H$. If $A_x\ne\est$, then this implies $g_0\in
x+y+z+H=A_x+B_y+C_z\seq A+B+C$, contrary to the choice of $g_0$. This proves
i), and in the same way one obtains ii) and iii).
\end{proof}

\subsection{Recovering the structure.}

We prove the Main Lemma in a series of claims sharing all the assumptions and
notation of Section~\refs{setup}.

\begin{alphaclaim}\label{m:A}
Suppose that $x,y,z\in G$ satisfy $x+y+z\equiv g_0\pmod H$ and
$C_x^*\ne\est$. Then for any permutation $(U,V,W)$ of the trio $(A,B,C)$ such
that $V_y,W_z\ne\est$, we also have $V_z,W_y\ne\est$, while
$U_y=U_z=B_y^*=B_z^*=\est$. In addition,
\begin{align*}
  |\o{A^*_y}|+|\o{A^*_z}|+2|\o{U_x}| &\ge |H|
\intertext{and}
  |\o{A^*_y}|+|\o{A^*_z}|+4|\o{U_x}| &\ge 2|H|.
\end{align*}
\end{alphaclaim}

\begin{proof}
If we had $x\equiv y\pmod H$, then \refe{tx3} would give
  $$ \rho \ge \rho_x \ge 3|H|-(|U_x|+|V_x|+|W_x|), $$
while $|U_x|+|V_x|=|U_x|+|V_y|\le|H|$ by Lemma~\refl{box} (as $W_z\ne\est$).
It would then follow that $\rho\ge |H|$, contradicting~\refe{rho}.

Switching the roles of $y$ and $z$ and of $V$ and $W$ in this argument, we
similarly rule out the situation where $x\equiv z\pmod H$. Thus, we actually
have $x\not\equiv y\pmod H$ and $x\not\equiv z\pmod H$.

If we had $U_z=V_z=\est$, then from \refe{tx3} and \refe{tx1} we would obtain
  $$ \rho_x \ge 3|H|-|U_x|-|V_x|-|W_x| \ge |H|-|U_x| $$
and
  $$ \rho_z \ge |H|-|U_z|-|V_z|-|W_z| = |H|-|W_z|, $$
while $|U_x|+|W_z|\le |H|$ by Lemma~\refl{box}; consequently,
$\rho\ge\rho_x+\rho_z \ge |H|$, contradicting~\refe{rho}. Thus, at least one
of $U_z$ and $V_z$ is non-empty.

If \emph{both} $U_z$ and $V_z$ were non-empty, then we would get a
contradiction from
  $$ |U_x|+|V_y|\le|H|,\quad |V_x|+|W_y|\le|H|,\quad |W_x|+|U_y|\le|H| $$
(by Lemma~\refl{box}) and
  $$ \rho \ge \rho_x+\rho_y
                  \ge (3|H|-|U_x|-|V_x|-|W_x|)+(|H|-|U_y|-|V_y|-|W_y|) $$
(by~\refe{tx3} and~\refe{tx1}). It follows that \emph{exactly one} of $U_z$
and $V_z$ is empty. Switching the roles of $y$ and $z$ and of $V$ and $W$, in
the very same way we conclude that exactly one of $U_y$ and $W_y$ is empty.

We now claim that, indeed, $V_z$ and $W_y$ are non-empty, while $U_y$ and
$U_z$ are empty, for if, say, we had $V_z=\est$, then from
  $$ |U_x|+|W_z|\le|H|,\quad |U_z|+|W_x|\le|H| $$
(Lemma~\refl{box}) we would get
  $$ \rho \ge \rho_x+\rho_z \ge (3|H|-|U_x|-|V_x|-|W_x|) + (|H|-|U_z|-|W_z|)
                                                   \ge 2|H|-|V_x| \ge |H|, $$
and in a similar way we get a contradiction with~\refe{rho} assuming that
$W_y=\est$.

We have thus shown that $V_z,W_y\ne\est$ and $U_y=U_z=\est$. Now, if we had
$B_y^*\ne\est$, then in view of $U_y=\est$ this would result in
  $$ \rho \ge \rho_x+\rho_y \ge (3|H|-|U_x|-|V_x|-|W_x|)+(2|H|-|V_y|-|W_y|) $$
which, in conjunction with $|U_x|+|W_y|\le|H|$ and
 $|V_x|+|W_x|+|V_y|\le 3|H|$, contradicts~\refe{rho}. In the same way we
obtain a contradiction assuming $B_z^*\ne\est$. Thus, $B_y^*=B_z^*=\est$.

Finally, since
  $$  |U_x|+|W_z|\le|H|,\quad |U_x|+|V_y|\le|H|,\quad |W_y|+|V_z|\le|H| $$
by Lemma~\refl{box}, it follows in view of \refe{tx2} that
\begin{multline*}
  \quad |\o{A^*_y}|+|\o{A^*_z}|+2|\o{U_x}| \\
     \ge (|H|-|V_y|-|W_y|) + (|H|-|V_z|-|W_z|) + 2(|H|-|U_x|) \ge |H|\quad
\end{multline*}
and similarly, from
  $$  |U_x|+|V_y|\le|H|,\quad |U_x|+|W_y|\le|H|
                        ,\quad |U_x|+|V_z|\le|H|,\quad |U_x|+|W_z|\le|H|, $$
we get
\begin{multline*}
  \quad |\o{A^*_y}|+|\o{A^*_z}|+4|\o{U_x}| \\
     \ge (|H|-|V_y|-|W_y|) + (|H|-|V_z|-|W_z|) + 4(|H|-|U_x|) \ge 2|H|. \quad
\end{multline*}
\end{proof}

\begin{alphaclaim}\label{m:B}
There is at most one coset $x+H$ such that $C^*_x$ is partial. Moreover, if
$C^*_x$ is partial, then  exactly one of $A_x$, $B_x$ and $C_x$ is partial
while the other two are full.
\end{alphaclaim}

\begin{proof}
Assume by contradiction that $C^*_x$ and $C^*_\xi$ are both partial with
$x\not\equiv\xi\pmod H$. Since $C^*_x$ is partial, all three slices $A_x$,
$B_x$ and $C_x$ are nonempty with at least one of them partial. Let $(U,V,W)$
be a permutation of $(A,B,C)$ such that $U_x$ is partial.  Likewise, all
three slices $A_\xi$, $B_\xi$ and $C_\xi$ are nonempty with at least one
partial. Let $(U',V',W')$ be a permutation of $(A,B,C)$ such that $U'_\xi$ is
partial.

Recalling the observation above Lemma \refl{emptystar}, let $(U_x,V_y,W_z)$
and $(U'_\xi,V'_\eta,W'_\zet)$ be triples of partial (in particular,
nonempty) slices with
\begin{equation}\label{e:cong}
  x+y+z \equiv \xi+\eta+\zet \equiv g_0 \pmod H.
\end{equation}
Without loss of generality, we assume that $|\o{U'_\xi}|\ge|\o{U_x}|$.

We have $B^*_x\supseteq C^*_x\ne\est$ and  $B^*_\xi\supseteq C^*_\xi\ne\est$
while $B^*_y=B^*_z=B^*_\eta=B^*_\zet=\est$ by Claim~\refm{A}, and it follows
that $x$ and $\xi$ are distinct modulo $H$ from each of $y,z,\eta,\zet$.
Consequently, by~\refe{tx3}, \refe{tx2}, and the second inequality in
Claim~\refm{A},
\begin{multline*}
   \rho \ge \rho_x+\rho_\xi+\max\{\rho_y,\rho_z\}
      \ge |\o{U_x}|+|\o{U'_\xi}|+\max\{|\o{A^*_y}|,|\o{A^*_z}|\} \\
           \ge 2|\o{U_x}|+\frac12\big(|\o{A^*_y}|+|\o{A^*_z}|\big) \ge |H|.
\end{multline*}
This contradicts~\refe{rho}, showing that there is at most one coset $x+H$
such that $C^*_x$ is partial.

To complete the proof, we now show that, if $C^*_x$ is partial, then exactly
one of $A_x$, $B_x$, and $C_x$ is partial; since $C^*_x\neq\est$ ensures that
all three slices $A_x$, $B_x$ and $C_x$ are nonempty, this will also show
that the other two slices are full. For a contradiction, suppose that
$(U,V,W)$ is a permutation of $(A,B,C)$ such that $U_x$ and $V_x$ are both
partial and find then $y,z,\eta,\zet$ satisfying
  $$ x+y+z \equiv x+\eta+\zet \equiv g_0 \pmod H $$
so that all the components of $(U_x,V_y,W_z)$ and $(V_x,W_\eta,U_\zet)$ are
partial. As above, from Claim~\refm{A} we derive that $x$ is distinct modulo
$H$ from each of $y,z,\eta,\zet$. Furthermore, by Claim~\refm{A}, the unique
empty slice in $(U_y,V_y,W_y)$ is $U_y$, the unique empty slice in
$(U_z,V_z,W_z)$ is $U_z$, the unique empty slice in $(V_\eta,W_\eta,U_\eta)$
is $V_\eta$, and the unique empty slice in $(V_\zet,W_\zet,U_\zet)$ is
$V_\zet$;
it follows that $y$ is distinct modulo $H$ from each of $\eta$ and $\zet$,
and similarly $z$ is distinct modulo $H$ from each of $\eta$ and $\zet$.
Also, from \refe{tx1} and Claim~\refm{A},
  $$ |\o{U_x}|+\max\{\rho_y,\rho_z\} \ge |\o{U_x}|
        + \frac12\big(|\o{A^*_y}|+|\o{A^*_z}|\big) \ge \frac12|H| $$
and
  $$ |\o{V_x}|+\max\{\rho_\eta,\rho_\zet\} \ge |\o{V_x}|
        + \frac12\big(|\o{A^*_\eta}|+|\o{A^*_\zet}|\big) \ge \frac12|H|. $$
Since $\rho_x\ge|\o{U_x}|+|\o{V_x}|$ by \refe{tx3}, we derive that
  $$ \rho \ge \rho_x + \max\{\rho_y,\rho_z\}
                                  + \max\{\rho_\eta,\rho_\zet\} \ge |H|, $$
contradicting~\refe{rho}.
\end{proof}

\begin{alphaclaim}\label{m:C}
If $U=A+a$, $V=B+b$, and $W=C+c$, with $a,b,c,\in H$, then letting
$(U^*,V^*,W^*):=\tau(U,V,W)$, we have
  $$ A^*\seq U^*+H,\quad B^*\seq V^*+H \quad\text{and}\quad C^*\seq W^*+H. $$
Also,
  $$ A^*+B^*+C^* \seq U^*+V^*+W^*. $$
\end{alphaclaim}

\begin{proof}
The first assertion can be equivalently restated as follows: if, for a group
element $x$, some of the slices $A^*_x$, $B^*_x$, and $C^*_x$ are non-empty,
then the corresponding slices from among $U^*_x$, $V^*_x$, and $W^*_x$ are
non-empty, too. Let $t_x$ be the number of slices from among $A^*_x$, $B^*_x$
and $C^*_x$ that are non-empty. Then~\refe{rho} and~\refe{tx1} give
  $$ |H| > \rho \ge t_x|H|-\big(|A_x|+|B_x|+|C_x|\big), $$
which further leads to $|U_x|+|V_x|+|W_x|=|A_x|+|B_x|+|C_x|>(t_x-1)|H|$;
consequently, the pigeonhole principle ensures that at least $t_x$ slices
from among $U^*_x$, $V^*_x$ and $W^*_x$ are nonempty, and since
$W^*_x\subseteq V^*_x\subseteq U^*_x$ and $C^*_x\subseteq B^*_x\subseteq
A^*_x$ by definition of $\tau$, the claimed result follows.

We proceed to prove the inclusion $A^*+B^*+C^*\seq U^*+V^*+W^*$. Assuming for
a contradiction that it fails to hold, there exists a coset $g_1+H$ contained
in $A^*+B^*+C^*$
 but not in $U^*+V^*+W^*$. Find group elements $x$, $y$, and
$z$ with $x+y+z\equiv g_1\pmod H$ such that $(A^*_x,B^*_y,C^*_z)$, and
therefore also $(U^*_x,V^*_y,W^*_z)$,
 has all its components non-empty.
Since $U^*_x+V^*_y+W^*_z\subsetneq g_1+H$ and $U^*_x\ne\est$, the pigeonhole
principle gives
 $|V^*_y|+|W^*_z|\le|H|$, and hence
\begin{equation}\label{e:locB}
  |\o{V^*_y}| + |\o{W^*_z}| \ge |H|.
\end{equation}

If $C^*_z$ were full, then all of $A_z,B_z,C_z$, and consequently $W^*_z$,
would be full, contradicting $U^*_x+V^*_y+W^*_z\subsetneq g_1+H$. Thus
$C^*_z$ is partial, and by Claim~\refm{B}, two of the slices $A_z$, $B_z$,
and $C_z$ are full. As a result, using \refe{tx3} we obtain
\begin{equation}\label{e:locBa}
  \rho_z \ge |\o{C^*_z}| = |\o{W^*_z}|,
\end{equation}
and we also conclude that $B^*_z$ and $V^*_z$ both are full. Consequently, if
we had $W^*_y\ne\est$, this would result in
  $$ g_1+H = U^*_x+V^*_z+W^*_y \seq U^*+V^*+W^*, $$
a contradiction; thus, $W^*_y=\est$, and comparing this to $W^*_z\ne\est$,
we obtain $y\not\equiv z\pmod H$. Since $B^*_y\ne\est$ and $W^*_y=\est$, from
\refe{tx1} we now get
\begin{multline*}
  \rho_y \ge 2|H|-(|A_y|+|B_y|+|C_y|) = 2|H|-(|U_y|+|V_y|+|W_y|) \\
       = 2|H|-(|U^*_y|+|V^*_y|+|W^*_y|) \ge |H|-|V^*_y| \ge |\o{V^*_y}|.
\end{multline*}
In view of~\refe{locBa} and~\refe{locB}, this yields
  $$ \rho\ge \rho_y+\rho_z\ge |\o{V^*_y}| + |\o{W^*_z}|\ge |H|, $$
contradicting~\refe{rho}.
\end{proof}

For a set $S\seq G$, by $\<S\>$ we denote the subgroup of $G$ generated by
$S$. Thus, $\<S-S\>$ is the smallest subgroup $H\le G$ such that $S$ lies in
an $H$-coset.
\begin{alphaclaim}\label{m:D}
We have $H\le\pi(C^*)$; that is, $C^*$ is a union of $H$-cosets.
\end{alphaclaim}

\begin{proof}
If the assertion is wrong, then, by Claim~\refm{B}, there is a unique coset
$z+H$ such that $C^*_z$ is partial; moreover, of the three slices $A_z$,
$B_z$, and $C_z$, one is partial while the other two are full. To begin with,
we show that $C_z$ partial, whereas $B_z$ and $C_z$ are full.

Aiming at a contradiction, assume that, for instance, $B_z$ is partial, and
therefore there exist $x,y\in G$ with $x+y+z\equiv g_0\pmod H$ such that
$(A_x,B_z,C_y)$ has all its components non-empty. Observing that $A_y\ne\est$
by Claim~\refm{A}, fix arbitrarily an element $a\in A_y-C_y\seq H$. Letting
$(U^*,V^*,W^*):=\tau(A-a,B,C)$, we have then $U^*_x\ne\est$ (as $A_x\ne\est$),
$V^*_y\ne\est$ (as $(A_y-a)\cap C_y\ne\est$), and $W^*_z\ne\est$ (by
Claim~\refm{C}). Hence,
\begin{equation}\label{e:locD1}
  (U^*+V^*+W^*) \cap (g_0+H) \ne\est,
\end{equation}
whereas we know that
\begin{equation}\label{e:locD2}
   (A^*+B^*+C^*) \cap (g_0+H)=\est.
\end{equation}
Since
\begin{equation}\label{e:locD3}
  A^*+B^*+C^* \seq U^*+V^*+W^*
\end{equation}
by Claim~\refm{C}, this contradicts minimality of $|G\stm(A^*+B^*+C^*)|$,
unless $(U^*,V^*,W^*)=(A-a,B,C)$; that is, unless $C\seq B\seq A-a$. This,
however, is inconsistent with the assumption that $B_z$ is partial and $C_z$
is full.

We have shown that $B_z$ cannot be partial, and a similar argument shows that
neither can $A_z$. Consequently, $C_z$ is partial while both $A_z$ and $B_z$
are full, and we now re-use the argument above in these new settings.

Since $C_z$ is partial, there exist $x,y\in G$ with $x+y+z\equiv g_0\pmod H$
such that $(A_x,B_y,C_z)$ has all its components non-empty. Let $X$ be a set
of representatives modulo $H$ for all possible such choices of $x$ and
likewise let $Y$ be a set of representatives modulo $H$ for all such choices
of $y$.
By Claim~\refm{A}, for every pair $(x,y)\in X\times Y$ with $x+y+z\equiv
g_0\pmod H$, we have $B_x,A_y\ne\est$; hence, $X$ and $Y$ coincide modulo
$H$, and we can assume that, indeed, $X=Y$ holds.

Fix $(x,y)\in X\times Y$ with $x+y+z\equiv g_0\pmod H$, and suppose that
$b\in B_y-A_y\seq H$. Letting $(U^*,V^*,W^*):=\tau(A,B-b,C)$, we have then
$U^*_x\ne\est$ (as $A_x\ne\est$), $V^*_y\ne\est$ (as
 $(B_y-b)\cap A_y\ne\est$), and $W^*_z\ne\est$ (by Claim~\refm{C}). Hence,
\refe{locD1} holds true, and comparing it with \refe{locD2} and \refe{locD3},
we get $(U^*,V^*,W^*)=(A,B-b,C)$, implying $C\seq B-b\seq A$ --- for
otherwise the minimality of $|G\stm(A^*+B^*+C^*)|$ would be contradicted.
Likewise, for $a\in A_y-B_y\seq H$, letting $(U^*,V^*,W^*):=\tau(A-a,B,C)$,
we have then $U^*_x\ne\est$ (as $A_x\ne\est$), $V^*_y\ne\est$ (as
 $(A_y-a)\cap B_y\ne\est$), and $W^*_z\ne\est$ (by Claim~\refm{C}). Hence,
\refe{locD1} holds true, and comparing it with \refe{locD2} and \refe{locD3},
we get $(U^*,V^*,W^*)=\tau(A-a,B,C)$, implying $C\seq B\seq A-a$. To
summarize, for each $y\in Y$ and each $b\in B_y-A_y$  and $a\in A_y-B_y$, we
have
\begin{equation}\label{e:locDa}
  C\seq B-b\seq A \quad\text{and}\quad C\seq B\seq A-a.
\end{equation}
As a corollary, $B-B_y+A_y\seq A$, implying $A_y+B_\eta-B_y\seq A_\eta$ for
all $y,\eta\in Y$. Switching the roles of $y$ and $\eta$, we also get
$A_\eta+B_y-B_\eta\seq A_y$, and as a result,
  $$ A_y + B_\eta-B_\eta + B_y-B_y \seq A_y. $$
Letting $K:=\sum_{\eta\in Y} \<B_\eta-B_\eta\>$,  we conclude in view of
$Y=X$ that $K\le\pi(A_x)$ for each $x\in X$. From \eqref{e:locDa}, we also
see that
  $$ C\subseteq A\cap B. $$
Thus $C=A\cap B\cap C=C^*$. By Claim~\refm{B}, the set $C$ has then exactly
one partial slice; namely, $C_z$. It follows that all non-trivial triples
$(A_\xi,B_\eta,C_\zet)$ with $\xi+\eta+\zet\equiv g_0\pmod H$ have
$\zet\equiv z\pmod H$, and therefore have $\xi\in X$. Since the above-defined
subgroup $K$ lies below the period of each set $A_\xi$ with $\xi\in X$, it
must also lie below the period of $(A+B+C)\cap(g_0+H)=(g_0+H)\stm\{g_0\}$.
This, however, is only possible if $K=\{0\}$, forcing $|B_\eta|=1$ for each
$\eta\in Y$.

Let  $(A_\xi,B_\eta,C_z)$ be a nontrivial triple with $(\xi,\eta)\in X\times
Y$ and $\xi+\eta+z\equiv g_0\pmod H$. Since $B_z$ is full while $|B_\xi|=1$,
we have $z\not\equiv\xi\pmod H$, whence
  $$ \rho \ge \rho_z + \rho_\xi
       \ge (|H|-|C_z|) + (|H|-|A_\xi|-|B_\xi|-|C_\xi|) $$
by \refe{tx3} and \refe{tx1}. Since $|B_\xi|=1$, and $C^*_z\ne\est$ yields
$C_\xi=\est$ by Claim~\refm{A}, we conclude that
  $$ \rho \ge 2|H|-1 - (|C_z|+|A_\xi|), $$
and to obtain a contradiction with \refe{rho} and complete the proof it
remains to notice that $|C_z|+|A_\xi|\le |H|$ by Lemma~\refl{box}.
\end{proof}

\begin{alphaclaim}\label{m:E}
If, for some $y\in G$, at least two among the slices $A_y$, $B_y$, and $C_y$
are non-empty, then also $B^*_y$ is non-empty.
\end{alphaclaim}

\begin{proof}
Suppose for a contradiction that $B^*_y=\est$. By Lemma~\refl{emptystar},
there exist $x,z\in G$ with $x+y+z\equiv g_0\pmod H$ such that $A^*_x$ and
$C^*_z$ are non-empty. As a result, at least one of $A_x$, $B_x$, and $C_x$
is non-empty (as $A^*_x\ne\est$), at least two of $A_y$, $B_y$, and $C_y$ are
non-empty (by the assumption of the claim), and all three slices $A_z,B_z$,
and $C_z$ are non-empty (as follows from $C^*_z\ne\est$). Consequently, there
is a permutation $(U,V,W)$ of the original trio $(A,B,C)$ such that $U_x$,
$V_y$, and $W_z$ are all non-empty. Moreover, by Claim~\refm{D}, from
$C^*_z\ne\est$ it follows that $A_z$, $B_z$, and $C_z$ are full. In
particular, $W_z$ is full, and so $g_0+H=U_x+V_y+W_z\seq A+B+C$, a
contradiction.
\end{proof}

\begin{alphaclaim}\label{m:F}
Let $Z$ be a set of representatives of all those cosets $z+H$ such that
$A_z,B_z,C_z\ne\est$ but $C^*_z=\est$. Assuming that $Z\ne\est$, let
$K_B:=\sum_{z\in Z}\<B_z-B_z\>$ and $K_C:=\sum_{z\in Z}\<C_z-C_z\>$. Then for
each $z\in Z$, we have $K_B\le\pi(A_z)$ and $K_C\le\pi(B_z)\cap \pi(A_z)$.
\end{alphaclaim}

\begin{proof}
Fix $z\in Z$. By Lemma~\refl{emptystar}, there exist $x,y\in G$ with
$x+y+z\equiv g_0\pmod H$ and $A^*_x,B^*_y\ne\est$. Furthermore, for each
$b\in B_z-A_z\seq H$, we have $(B_z-b)\cap A_z\ne\est$, and we can find
 $c\in H$ so that, indeed, $A_z\cap(B_z-b)\cap(C_z-c)\ne\est$. Letting
$(U^*,V^*,W^*):=\tau(A,B-b,C-c)$, we thus have $W^*_z\ne\est$ and, by
Claim~\refm{C}, we have $U^*_x,V^*_y\ne\est$. This shows that
  $$ (U^*+V^*+W^*)\cap(g_0+H) \,\ne\, (A^*+B^*+C^*)\cap(g_0+H)\, =\, \est, $$
and since $A^*+B^*+C^*\seq U^*+V^*+W^*$ by Claim~\refm{C}, the minimality of
the quantity
$|G\stm(A^*+B^*+C^*)|$ implies $(U^*,V^*,W^*)=(A,B-b,C-c)$; that is,
 $C-c\seq B-b\seq A$. Recalling that $b$ was chosen to be an arbitrary element
of $B_z-A_z$, we conclude that $B+A_z-B_z\seq A$, and in particular,
$A_z+B_\zet-B_z\seq A_\zet$ for any $\zet\in Z$. Switching the roles of $z$
and $\zet$, we also get $A_\zet+B_z-B_\zet\seq A_z$, and combining these
inclusions, we obtain $A_z+(B_z-B_z)+(B_\zet-B_\zet)\seq A_z$. As a result,
$K_B\le\pi(A_z)$, as required.

The second assertion follows in a similar way: for each $c\in C_z-B_z$,
there exists $a\in H$ with $(C_z-c)\cap B_z\cap(A_z-a)\ne\est$, and then the
minimality of $|G\stm(A^*+B^*+C^*)|$ gives $C-c\seq B\seq A-a$; this shows
that $C_\zet+B_z-C_z\seq B_\zet$ for all $z,\zet\in Z$, and combining this
with $C_z+B_\zet-C_\zet\seq B_z$ yields $K_C\leq \pi(B_z)$.

The final portion of the claim also follows in a similar way. For each
$c\in C_z-A_z$, there exists $b\in H$ with
$(C_z-c)\cap (B_z-b)\cap A_z\neq \emptyset$, and then the minimality of
$|G\setminus (A^*+B^*+C^*)|$ gives $C-c\subseteq B-b\subseteq A$. This
shows that $C_\zeta+A_z-C_z\subseteq A_\zeta$ for all $z,\zeta\in Z$, and
combining this with $C_z+A_\zeta-C_\zeta\subseteq A_z$ yields
$K_C\subseteq \pi(A_z)$.
\end{proof}

\begin{corollary}\label{c:F}
If, for some $z\in G$, we have $A_z,B_z,C_z\ne\est$ while $C^*_z=\est$, then
each of $B_z$ and $C_z$ is contained in a coset of $\pi(A_z)$.
\end{corollary}

\begin{proof}
Let $K_B,K_C\le H$ be as in Claim~\refm{F}. Then
 $\<B_z-B_z\>\le K_B\le\pi(A_z)$ shows that $B_z$ is contained in a coset of
$\pi(A_z)$, and then from $\<C_z-C_z\>\le K_C\le\pi(A_z)$ we derive that
$C_z$ is contained in a coset of $\pi(A_z)$.
\end{proof}

\begin{alphaclaim}\label{m:G}
If, for some $y\in G$, the slice $B_y$ is partial, then $A_y$ and $C_y$ are
both non-empty while $C^*_y$ is empty.
\end{alphaclaim}

\begin{proof}
Since $B_y$ is partial, there exist $x,z\in G$ such that
 $x+y+z\equiv g_0\pmod H$ and $A_x,C_z$ are both partial. We notice that none
of the slices $C^*_x,C^*_y$, and $C^*_z$ is full; hence by Claim~\refm{D},
all of them are actually empty.

Suppose that $w\in\{x,y,z\}$. If exactly one of the slices $A_w,B_w$, and
$C_w$ is non-empty, then denoting by $U$ the corresponding set from among
$A,B$, and $C$, we have
  $$ \rho_w = |\o{A^*_w}| = |\o{U_w}|. $$
If exactly two of $A_w,B_w$, and $C_w$ are non-empty, then $B^*_w\ne\est$ by
Claim~\refm{E}, and denoting by $U$ and $V$ the sets from among $A,B$, and
$C$ corresponding to the non-empty slices, we have
  $$ \rho_w = |\o{A^*_w}| + |\o{B^*_w}| = |\o{U_w}| + |\o{V_w}|. $$
If $A_w,B_w$, and $C_w$ are all non-empty with $A_w$ partial, then
$B^*_w\ne\est$ by Claim~\refm{E}, and by Corollary~\refc{F}, each of $B_w$
and $C_w$ lies in a coset of $\pi(A_w)$ (recall that $C^*_w=\emptyset$ as
noted at the beginning of the proof); consequently,
\begin{equation}\label{e:locG}
  \rho_w = 2|H|-|A_w|-|B_w|-|C_w|
            \ge 2|H| - 2|\pi(A_w)| - (|H|-|\pi(A_w)|) \ge \frac12\,|H|.
\end{equation}
With these preliminary observations, we can now prove that $A_y$ and $C_y$
are non-empty.

If $x\equiv y\equiv z\pmod H$, then the assertion is immediate as we have
chosen $x$ and $z$ so that $A_x,C_z\ne\est$. Assume now that $x,y$, and $z$
all are different modulo $H$, and, for a contradiction, that there are at
most two non-empty slices among $A_y,B_y,C_y$. As we have shown above, the
latter assumption implies $\rho_y\ge|\o{B_y}|$ (as $B_y$ is nonempty by
hypothesis). Also, if there were at most two non-empty slices among $A_x$,
$B_x$ and $C_x$, then we would have $\rho_x\ge|\o{A_x}|$, whence
Lemma~\refl{box} yields
  $$ \rho\ge\rho_x+\rho_y\ge |\overline{A_x}|+|\overline{B_y}|
                                          = 2|H|-(|A_x|+|B_y|) \ge |H|, $$
which contradicts~\refe{rho}. Thus, $A_x,B_x$, and $C_x$ are all non-empty,
and in a similar way, $A_z,B_z$, and $C_z$ are all non-empty. Now
$A_x+B_y+C_z\ne g_0+H$ shows that $A_x$ is partial, and $A_z+B_y+C_x\ne
g_0+H$ shows that $A_z$ is partial. Hence, \refe{locG} yields
  $$ \rho \ge \rho_x+\rho_z \ge |H|, $$
contradicting~\refe{rho}.

We have thus shown that exactly two of $x,y$, and $z$ coincide modulo
$H$. If $x\equiv y\not\equiv z\pmod H$, then $A_y=A_x\ne\est$ whence,
assuming $C_y=\est$, we would get
  $$ \rho \ge \rho_y = |\o{A_y}|+|\o{B_y}| = 2|H|-|A_x|-|B_y| \ge |H| $$
by Lemma~\refl{box}. In a similar way we obtain a contradiction if
$x\not\equiv y\equiv z\pmod H$, and it remains to consider the case where
$x\equiv z\not\equiv y\pmod H$. If in this case $B_x=B_z=\est$, then we
obtain a contradiction from
  $$ \rho\ge \rho_x+\rho_y\ge|\o{A_x}|+|\o{B_y}|
                                          = 2|H|-|A_x|-|B_y| \ge |H|, $$
the last estimate following by Lemma~\refl{box}.

Assume therefore that $B_x=B_z\ne\est$. In this case $A_x$, $B_x$, and
$C_x=C_z$ are all non-empty while $C^*_x=\est$, whence
\begin{equation}\label{e:locG98}
  |B_x|\le|\pi(A_x)|
\end{equation}
by Corollary~\refc{F}. On the other hand,
\begin{equation}
  |A_x| \le |H| - |\pi(A_x)|
\end{equation}
since $A_x$ is partial (see the beginning of the proof). Furthermore,
$A_x,B_x,C_x\ne\est$ yields $B^*_x\ne\est$ by Claim~\refm{E}, implying
\begin{equation}
  \rho_x \ge 2|H|-|A_x|-|B_x|-|C_x|
\end{equation}
in view of~\refe{tx1}, while
\begin{equation}
  \rho_y \ge |\o{B_y}| = |H|-|B_y|
\end{equation}
as follows from an observation at the beginning of the proof. Finally,
\begin{equation}\label{e:locG99}
  |B_y|+|C_x| = |B_y|+|C_z| \le |H|
\end{equation}
by Lemma~\refl{box}. Combining~\refe{locG98}--\refe{locG99}, we get
\begin{multline*}
  \rho\ge\rho_x+\rho_y \ge (2|H|-|A_x|-|B_x|-|C_x|)+(|H|-|B_y|) \\
    = 3|H| - (|B_y|+|C_x|)-(|A_x|+|B_x|) \ge |H|,
\end{multline*}
contradicting~\refe{rho}. This shows that $A_y$ and $C_y$ are both nonempty,
and now $C^*_y=\emptyset$ follows from Claim D, else $B_y$ would be full,
contrary to hypothesis.
\end{proof}

\subsection{Conclusion of the proof.}

We are ready to complete the proof of the Main Lemma.

Let  $y\in G$ be an arbitrary element such that the slice $B_y$ is partial.
(Notice that such elements exist since otherwise $B$ would be $H$-periodic,
while we assume that $(A,B,C)$ is an aperiodic trio.) By Claim~\refm{G}, we
have $A_y,C_y\ne C^*_y=\est$, and keeping the notation $Z$, $K_B$, and $K_C$
of Claim~\refm{F}, we then conclude that $y\in Z$ and $K_C\le\pi(B_y)$. Thus,
every partial slice of $B$ is $K_C$-periodic, and it follows that $B$ itself
is $K_C$-periodic, implying $K_C=\{0\}$; that is, $|C_z|=1$ for each
 $z\in Z$. In particular, $|C_y|=1$.

If now $A_z$ were not full for some $z\in Z$, then we would have
$|A_z|\le|H|-|K_B|$ by Claim~\refm{F}, and since $B^*_z$ is non-empty by
Claim~\refm{E}, using~\refe{tx1} we would obtain
  $$ \rho \ge \rho_y \ge 2|H|-|A_z|-|B_z|-|C_z|
                          \ge 2|H| - (|H|-|K_B|) - |K_B| - 1 = |H|-1, $$
contradicting~\refe{rho}. Thus $A_z$ is full for every $z\in Z$. In
particular, $A_y$ is full and
  $$ \rho_y \ge  2|H|-|A_y|-|B_y|-|C_y| = |H|-|B_y|-1. $$

Since $B_y$ is partial, we can find $x,z\in G$ with $x+y+z\equiv g_0\pmod H$
and $A_x,C_z\ne\est$. By Claim~\refm{D}, for each $w\in\{x,y,z\}$ we have
$C^*_w=\est$: for otherwise $A_w$, $B_w$, and $C_w$ all would be full,
leading to $A_x+B_y+C_z=g_0+H$. Consequently, there is at least one empty
slice among $A_x,B_x$, and $C_x$: else $x\in Z$ and (as we have just shown)
$A_x$ would then be  full, whence $A_x+B_y+C_z=g_0+H$. Since, in contrast,
$A_y$, $B_y$ and $C_y$ are all non-empty, we have $x\not\equiv y\pmod H$.
Furthermore, arguing as at the beginning of the proof of Claim~\refm{G}, we
get $$\rho_x\ge|\o{A_x}|.$$  We have shown that that $x\not\equiv y\pmod H$,
whence combining the above inequalities yields
  $$ \rho \ge \rho_x+\rho_y \ge |\o{A_x}|+(|H|-|B_y|-1)
                                      = 2|H|-1-(|B_y|+|A_x|) \ge |H|-1 $$
by Lemma~\refl{box}, which contradicts \eqref{e:rho} and thus completes the
proof.

\section{Concluding remarks.}\label{s:concrem}

In hindsight, the following stronger (and simpler) version of the Main Lemma
follows easily from Theorem~\reft{Krest}.
\begin{lemma}\label{l:IWML}
Suppose that $(A,B,C)$ is an aperiodic, maximal, deficient trio, and let
$(A^*,B^*,C^*):=\tau(A,B,C)$ and $H:=\pi(A^*+B^*+C^*)$. If $C^*\ne\est$, then
  $$ |(A^*+H)\stm A^*|+|(B^*+H)\stm B^*|+|(C^*+H)\stm C^*| \ge |H|-1. $$
\end{lemma}

Lemma~\refl{IWML} is the ``ideal-world main lemma''. To derive it from
Theorem~\reft{Krest}, notice that by Lemma~\refl{major-infi} and
Corollary~\refc{maxdef},
  $$ \del(A^*+H,B^*+H,C^*+H) \ge \del(A^*,B^*,C^*) = \del(A,B,C) = 1, $$
and that
  $$ \del(A^*+H,B^*+H,C^*+H) =
                          |H|\,\del(\phi_H(A^*),\phi_H(B^*),\phi_H(C^*)). $$
Thus, we have in fact
  $$ \del(A^*+H,B^*+H,C^*+H) \ge |H|, $$
which implies
  $$ \del(A^*+H,B^*+H,C^*+H)-\del(A^*,B^*,C^*) \ge |H|-1. $$
This is equivalent to the inequality of Lemma~\refl{IWML}.

It is a major challenge to give Lemma~\refl{IWML} a simple, independent
proof.

Interestingly, Theorem~\reft{Krest} is equivalent to the following statement:
\begin{equation}\label{e:dag}
  \text{For any \emph{maximal} trio $(A,B,C)$, we have}\ %
                                            \del(A,B,C) \le |\pi(A+B+C)|.
\end{equation}
To derive Theorem~\reft{Krest} from \refe{dag}, given a trio $(A,B,C)$,
construct $(A',B',C')$ as in Lemma~\refl{astrio}. Since $(A',B',C')$ is
maximal, applying \refe{dag} to it we get
  $$ \del(A,B,C)\le \del(A',B',C') \le |\pi(A'+B'+C')| \le |\pi(A+B+C)|. $$

An easy consequence of Theorem~\reft{Krest} is a characterization of
deficient trios as those which can be obtained be removing few elements from
a \emph{maximal} deficient trio.
\begin{claim}
A trio $(A,B,C)$ is deficient if and only if there exists a \emph{maximal}
deficient trio $(A',B',C')$ such that $A\seq A'$, $B\seq B'$, $C\seq C'$, and
\begin{equation}\label{e:locclaim2}
  |A'\stm A|+|B'\stm B|+|C'\stm C| < |\pi(A'+B'+C')|.
\end{equation}
\end{claim}

\begin{proof}
If  $(A',B',C')$ is maximal and $(A,B,C)$ satisfies \refe{locclaim2}, then by
Corollary~\refc{maxdef},
  $$ \del(A,B,C) = \del(A',B',C')
                       - (|A'\stm A|+|B'\stm B|+|C'\stm C|) > 0. $$
Conversely, given a deficient trio $(A,B,C)$, for the supertrio $(A',B',C')$
of Lemma~\refl{astrio} we have
\begin{multline*}
  \quad |A'\stm A|+|B'\stm B|+|C'\stm C| = \del(A',B',C') - \del(A,B,C) \\
       < \del(A',B',C') = |\pi(A'+B'+C')| \quad
\end{multline*}
(the last equality uses Corollary~\refc{maxdef} again).
\end{proof}

Finally, we note that Lemma~\refl{basic} can be extended to take into account
the number of representations of group elements.
\begin{lemma}
For any finite subsets $A_1\longc A_n$ of an abelian group, letting
$(A_1^*\longc A_n^*):=\tau(A_1\longc A_n)$, the number of representations of
any group element as $a_1^*\longp a_n^*$ (with $a_i^*\in A_i^*$ for each
$i\in[1,n]$) does not exceed the number of its representations as
 $a_1\longp a_n$ (with $a_i\in A_i$ for each $i\in[1,n]$).
\end{lemma}
We omit the proof.

\bigskip

\end{document}